\documentclass[final,leqno,letterpaper]{my_etna_preprint}
\usepackage{panayot}

\hypersetup{
    pdftitle={Composite Adaptive Algebraic Multigrid},
    pdfauthor={Austen Nelson, Panayot Vassilevski},
    pdfkeywords={composite adaptive algebraic multigrid, preconditioner, solver, anisotropy, spd, graph, network, modularity matching}
}

\title{Characterization of the near-null error components utilized in composite adaptive AMG solvers
\thanks{
This work was partially supported by the National Science Foundation through NSF RTG grant DMS-2136228. Computational resources used in this work were supported by the Oregon Regional Computing Accelerator (Orca), funded by NSF CC*-2346732.
}}

\author{Austen J. Nelson\footnotemark[2]
        \and Panayot S. Vassilevski\footnotemark[3]}

\shorttitle{Near-null components characterization in composite aAMG} 
\shortauthor{A.J.~NELSON AND P.S.~VASSILEVSKI}

\begin{document}

\maketitle

\renewcommand{\thefootnote}{\fnsymbol{footnote}}

\footnotetext[2]{ajn6@pdx.edu}
\footnotetext[3]{panayot@pdx.edu}

\begin{abstract}
	We provide a theoretical justification for the construction of adaptive composite solvers based on a sequence of AMG (algebraic multigrid) $\mu$-cycle methods that exploit error components that the current solver cannot damp efficiently. 
	Each solver component is an aggregation based AMG where its aggregates are constructed using the popular in graph community detection modularity matrix. The latter utilizes the given matrix and the error component vector the current solver cannot handle. The performance of the resulting adaptive composite solver is illustrated on a variety of sparse matrices both arising from discretized PDEs and ones with more general nature. 
\end{abstract}



\section{Introduction}
\label{section: introduction}
Adaptive AMG (or aAMG), also known as bootstrap AMG (or BAMG), gives examples of a purely algebraic strategy for designing linear solvers with guaranteed convergence property. They go back to publications by Achi Brandt and Steve McCormick and their coauthors, e.g., \cite{BrandtBAMG, BFMMMR2004, BFMMMR2006, BBKL2011}. 
The idea behind aAMG methods is exactly a bootstrap one, i.e., 'use the method to improve the method.'  More specifically, we test a given method B to detect components of the error of the process that are not reduced effectively. In the setting of linear systems with a matrix $A$, these are vectors $\bv$ such that the iteration matrix $E= I-B^{-1} A$ does not visibly reduce them, i.e.
$E \bv \approx \bv$, or equivalently $B^{-1}A \bv \approx 0$.

In the present paper, for the class of aAMG methods, namely the adaptive composite AMG introduced in \cite{DV2013}, we provide some theoretical justification for vectors $\bv$ such that $B^{-1}A \bv \approx 0$ actually satisfies $A \bv \approx 0$ for AMG-type solvers $B$ and their composition. That is, the error components $\bv$ for which the method $B$ applied to $A$ is not efficient, are actually nearly null components of the matrix $A$ (also called algebraically smooth components). This is an important characterization, since the AMG solvers that we construct in an adaptive process require (explicit) access only to the entries of the matrix $A$ and the constructed vector $\bv$ (and not the entries of the current solver inverse 
$B^{-1}$ since they are not feasibly accessible). 

We also illustrate the convergence behavior of the aAMG methods that are composed of a sequence of $\mu$ cycles, where each $\mu$ cycle is an aggregation AMG constructed on the basis of the given matrix $A$ and the algebraically smooth vector currently computed $\bv$. In the coarsening process, to construct the hierarchy of aggregates, we use a graph-based coarsening derived from the popular in the graph clustering literature a specially constructed modularity matrix that utilizes $A$ and the vector $\bv$. Some preliminary results for a two-level setting were previously reported in \cite{Quiring2019}.

In what follows, we consider an s.p.d. large sparse matrix $A = (a_{ij})^n_{i,j=1}$. 
We will be interested in the stationary iterative process for solving systems of equations $A \bx = \bb$ with an $A$-convergent method $B$. The method  $B$ is assumed to give rise to a s.p.d. mapping. That is, we assume that the iteration matrix $E = I-B^{-1}A$ satisfies $\|E \bv\|_A \le \varrho \|\bv\|_A$ for a $\varrho < 1$.  Here and in what follows, $\|\bw \|_A= \sqrt{\bw^TA\bw}$ stands for the standard $A$-norm.

\subsection*{Outline}
The remainder of the paper is structured as follows. In Section~\ref{section: composition of solvers}, we describe in general terms the composition of solvers and study some of their properties. 
In Section~\ref{section: Characterization of near null components}, we prove our result characterizing the near null components of $B^{-1}A$ as near null components of $A$ under some assumptions on $B$.
In Section~\ref{section: the adaptive process} we use the results up to this point to motivate and present the main algorithm used to construct composite adaptive operators.
In Section~\ref{section: modularity based AMG}, we describe one strategy to create an aggregation based V-cycle exploiting a near-null vector $\bv$. The vector $\bv$ for which $A \bv \approx 0$ is used to scale the matrix $A$ and its resulting scaled matrix entries are used to assign edge weights to the sparsity matrix graph (sometimes referred to as the `strength of connection' graph). These weights are then utilized in a recursive matching algorithm to construct aggregates which are aimed to follow the strong connections defined by the matrix entries and the algebraically smooth vector $\bv$. The latter is motivated and achieved in practice by the popular modularity based community detection, cf., Newman \cite{Newman}, and its most successful implementation referred to as the Louvain algorithm \cite{Blondel2008}. By optimizing the corresponding modularity functional, it is aimed to group vertices in a community that are strongly connected with weaker connections across communities. In our setting, aggregates serve as communities. That is, the aggregates that we construct respect the strength of connectivity in a similar way as in the Louvain community detection algorithm. We conclude this section with a mention to smoothed aggregation (SA-AMG) which provides a well studied recipe to incorporate nodal aggregates into the construction of interpolation, restriction, and coarse grid operators.
In Section~\ref{section: experiments}, we demonstrate the performance of the modularity based composite aAMG algorithm on a variety of difficult test problems corresponding to both PDE discretization matrices and to ones of more general nature.

\section{Symmetric composition of solvers}\label{section: composition of solvers}
Consider $A \bx = \bb$ with a sparse s.p.d. matrix $A$. Let $B_0$ be a $A$-convergent solver for $A$. In our main application $B_0$ could be one of the three choices: the 
$\ell_1$-smoother, the forward Gauss-Seidel one, and a symmetric $\mu$-cycle ML (multilevel) solver. 

Our goal is by testing a solver $B_0$ on $A \bx =0$ (starting with a nonzero initial iterate) in case a convergence stall occurs, to use the near-null component vector $\bw = \frac{\bx}{\|\bx\|}$, where $\bx:\; (I-B^{-1} A)\bx \approx \bx$ to construct a new solver $B_1 = B(\bw)$ that depends on the near-null component $\bw$, so that $B_1$ eliminates $\bw$ from the error. Our choice for $B_1=B(\bw)$ is an aggregation-based AMG which builds the interpolation matrices that have $\bw$ in their range. This is described in detail in Section~\ref{section: modularity based AMG}.

Once $B_1= B(\bw)$ is available, we combine the original solver $B_0$ with the newly created one using the following symmetric composition strategy.

Given two solvers $B_0$ and $B_1$ (and $B^T_1$ in the nonsymmetric case), we look at the product iteration matrix which defines a new solver $B$ from the equation

\begin{equation}\label{symmetric composition of two solvers: iteration matrix}
I- B^{-1} A = (I-B^{-T}_1 A) (I-B^{-1}_0A)(I-B^{-1}_1 A).
\end{equation}

The new solver $B$ combines $B_1, B_0$ and $B^T_1$. More explicitly, we have the formula

\begin{equation}\label{symmetric composition of two solvers}
B^{-1} = {\overline B}^{-1}_1 + (I-B^{-T}_1A) B^{-1}_0 (I-AB^{-1}_1).
\end{equation}

The solver ${\overline B}_1$ is a symmetrization of $B_1$, i.e., it corresponds to the iteration formula 

\begin{equation*}
I-{\overline B}^{-1}_1A = (I-B^{-T}_1A)(I-B^{-1}_1A),
\end{equation*}

or more explicitly 

\begin{equation}\label{symmetrized B_1}
{\overline B}_1 = B_1 (B_1+B^T_1-A)^{-1} B^T_1.
\end{equation}

Notice that ${\overline B}_1$ is s.p.d. for any solver $B_1$ that is $A$-convergent (cf., e.g., Proposition 3.8 in \cite{Vas08}). 

The following estimate holds

\begin{equation}\label{estimate of norms}
    \|B\| \le \|{\overline B}_1\|.
\end{equation}

It follows from the fact that the term

\begin{equation*}
  (I-B^{-T}_1A) B^{-1}_0 (I-AB^{-1}_1)
\end{equation*}

in the definition \eqref{symmetric composition of two solvers} of $B^{-1}$ is positive semi-definite, which implies
$\bv^T B^{-1} \bv \ge \bv^T {\overline B}^{-1}_1 \bv$ for any $\bv$ and hence 
$\bv^T B \bv \le \bv^T {\overline B}_1 \bv$, and as a corollary $\|B\| \le \|{\overline B}_1\|$.

In conclusion, we have the following result for the properties of the symmetric composition of solvers.
\begin{lemma}\label{lemma: properties of composite solvers}
    If $B_0$ is s.p.d. and $B_0$ and $B_1$ are $A$-convergent solvers, then their composition defined in \eqref{symmetric composition of two solvers: iteration matrix} or equivalently, in \eqref{symmetric composition of two solvers}, is s.p.d. and $B$ is also $A$-convergent. Also, if the symmetrized solver ${\overline B}_1$ (see \eqref{symmetrized B_1}) satisfies 
    $\|{\overline B}_1\| \le c_0\|A\|$ for some constant $c_0 >0$, then the same inequality holds for $B$, i.e., $\|B\| \le c_0\|A\|$ (due to \eqref{estimate of norms}).
    Finally, if $B_1$ is s.p.d. and satisfies the inequalities
    $\bv^T B_1 \bv \ge \bv^T A \bv$ and $\|B_1\| \le c_0 \|A\|$, we have $\|B\| \le \|{\overline B}_1\| \le \|B_1\| \le c_0 \|A\|$.
    \end{lemma} 
\begin{proof}
The only thing that remains to be shown is that if $B_1-A$ is symmetric positive semi-definite that $B_1-{\overline B}_1$ is symmetric positive semidefinite.
The latter is immediately seen from the fact that $X \equiv B^T_1+B_1-A = 2B_1 - A$, hence $X-B_1=B_1-A$ is symmetric positive semidefinite.
Therefore $$\bv^T {\overline B_1} \bv = \bv^T B_1 X^{-1} B_1 \bv \le \bv^T B_1 B^{-1}_1 B_1 \bv =\bv^T B_1 \bv.$$
\end{proof}

\section{Detecting slow to converge error components in iterative methods}\label{section: Characterization of near null components}
Given a $n \times n$ s.p.d. matrix $A$, starting with a convergent (in the $A$-norm) iterative method such as scaled Jacobi, the $\ell_1$ diagonal preconditioner, or (symmetric) Gauss-Seidel,
or a multilevel one, all  denoted by $B$, we test its convergence properties. 
In order to test their convergence quality, we consider the trivial equation
$ A \bx = 0$,
and starting with a nonzero (random) initial iterate 
$\bx_0$, for $k=1,2,\dots,$  we solve for $\bx_k$ 
\begin{equation}\label{iteration with B}
B (\bx_k -\bx_{k-1}) = - A \bx_{k-1},
\end{equation}
 until a stall occurs, i.e., for $k \ge m$, we have 
\begin{equation*}
\frac{\|\bx_k\|_A}{\|\bx_{k-1}\|_A} \ge 0.999...
\end{equation*}
We note that the vectors $-\bx_k$ are actually the errors (since the exact solution is $\bx=0$). 
The occurrence of a stall is an indication that the current iteration method based on $B$ is not able to reduce certain components of the error, namely the ones dominating the current iterate $\bx_k$. 
In other words, the  vector $\bw = \frac{\bx_m}{\|\bx_m\|_A}$ is such that
$B^{-1}A \bw \approx 0$, i.e., $\bw$ is spanned by eigenmodes of $B^{-1}A$ in the lower portion of the spectrum of 
$B^{-1}A$.

We have the following main result of the present paper characterizing solvers $B$ that can actually lead to vectors $\bw$  that are near-null components of $A$ (and not only of $B^{-1}A$). The latter is important since in the coarsening process we have access to the entries of $A$ and $\bw$ (and not necessarily to $B^{-1}$). 
Those include standard smoothers (Gauss-Seidel, scaled Jacobi, $\ell_1$) and also $\mu$-cycles utilizing these smoothers. 
\begin{theorem}\label{theorem: near null component}
Let $A$ be a $n \times n$ s.p.d. matrix. 
Let $B$ be symmetric positive definite such that 
\begin{equation*}
\bv^T A \bv \le \bv^T B \bv \text{ for any }
\bv \in {\mathbb R}^n,
\end{equation*}
so that $B$ defines an $A$-convergent iteration method  for solving $A \bx = \bb$.
We assume that $\|B\| \simeq \|A\|$, i.e., 
\begin{equation}\label{norm equivalence}
\|B\| \le c_0 \|A\| \text{ for a constant }c_0 \ge 1. 
\end{equation}
Consider any vector $\bw$ such that the iteration process \eqref{iteration with B}
with $B$ stalls for it, i.e.,
\begin{equation}\label{stall inequality}
 1 \ge \frac{\|(I- B^{-1}A)\bw\|^2_A}{\|\bw\|^2_A} \ge 1-\delta,
\end{equation}
 for some small $\delta \in (0,1)$. Then, the following estimate holds $\|A\bw\|^2 \le c_0 \|A\|\;\delta \|\bw\|^2_A.$
\end{theorem}
In other words, if we scale $A:\; \|A\|=1$ and $\bw:\; \|\bw\|_A = 1$,  then we have $\|A\bw\|_{\max} \le \|A\bw\| \le \sqrt{c_0 \delta} = \O(\sqrt{\delta})$, i.e., $A \bw \approx 0$ componentwise. 
That is, vectors for which the iteration process with $B$ stalls are in the near-nullspace of $A$ commonly referred to as {\em algebraically smooth} vectors generated by the solver $B$. 
\begin{proof}
From the inequality $(1-\delta) \|\bw\|^2_A \le \|(I-B^{-1}A) \bw\|^2_A$ we obtain
\begin{equation*}
\begin{array}{rl}
(1-\delta) \|\bw\|^2_A & \le \left ((I-B^{-1}A) \bw\right )^T A \left ((I-B^{-1}A) \bw\right )\\
& = \bw^TA \bw + \bw^TA B^{-1}  A  B^{-1} A \bw -2 \bw^T A B^{-1}A \bw\\
&= \bw^TA \bw  - \bw^T A B^{-1} A \bw - \bw^T (A B^{-1}A - A B^{-1}  A  B^{-1} A) \bw\\
&= \|\bw\|^2_A - \bw^T A B^{-1} A \bw - \bw^T AB^{-1} (B-A) B^{-1}A \bw\\
\end{array}
\end{equation*}
That is, we have
\begin{equation*}
\bw^T A B^{-1} A \bw + \bw^T AB^{-1} (B-A) B^{-1}A \bw \le \delta \|\bw\|^2_A.
\end{equation*}
Now, since $B-A$ is symmetric positive semi-definite, and using the inequality 
$$\frac{1}{\|B\|} \|A\bw\|^2 \le \bw^T A B^{-1} A \bw,$$  we arrive at
\begin{equation*}
\frac{1}{\|B\|} \|A\bw\|^2 \le \delta \|\bw\|^2_A.
\end{equation*}
Finally, using the assumption $\|B\| \le c_0\|A\|$ the final result follows 
$$ \|A\bw\|^2 \le c_0 \delta \|A\|\|\bw\|^2_A.$$

\end{proof}
\begin{proposition}\label{proposition: solvers satisfying norm equivalence condition}
Here, we list $A$-convergent s.p.d. solvers $B$ that satisfy the main assumption \eqref{norm equivalence}.
\begin{itemize}
    \item [(i)] "Standard"  smoothers: symmetric Gauss-Seidel and the $\ell_1$-smoother.
    \item [(ii)] Symmetric MG cycles using the above smoothers (i).
    \item [(iii)] Compositions of solvers $B_0$ and $B_1$  as studied in Lemma~\ref{lemma: properties of composite solvers}.  In particular, the composition of solvers $B_{composite}$  as described in the following section satisfies \eqref{norm equivalence}, which justifies the adaptive process in Section~\ref{section: the adaptive process}. 
\end{itemize}
\end{proposition}
\begin{proof}
The result in item (i) is well-known (e.g., \cite{Vas08}) and  follows from the fact that these smoothers are spectrally equivalent to the diagonal of the given matrix $A$ (valid for any sparse s.p.d. matrix with bounded number of non-zero entries per row). 
Items (ii) and (iii) follow from Lemma~\ref{lemma: properties of composite solvers}.
    \end{proof}

\section{The adaptive process}\label{section: the adaptive process}
Once the new solver $B$, 
\eqref{symmetric composition of two solvers}, is created, we repeat the procedure, i.e., we test if the new solver has good convergence or not. If it does not, we construct its near-null component vector and create a corresponding hierarchy and respective $\mu$-cycle solver that removes that vector from the error. 
Then, we make a new composite solver from the most recent one and the last constructed $\mu$-cycle one.
This process is repeated successively until a solver with a satisfactory convergence is built. It is clear that we cannot have more than $n$ such steps (where $n$ is the size of the original matrix $A$).
This strategy resembles variational iterative methods, where at each iteration we construct a search direction that is orthogonal to the previous search directions, hence  we cannot have more than $n$ orthogonal search directions. 
Given $A \bx = \bb$, and a desired convergence factor $\varrho \in (0,1)$, Algorithm \ref{alg:adaptive-ml} constructs a sequence of solvers, $B_0, B_1,\; \dots, B_m$, for  some $m \ge 0$, which symmetric composition, $B_{composite}$, leads to an iteration matrix 
\begin{equation*}
\begin{array}{rl}
E_{composite} &= I - B^{-1}_{composite}A\\
&= (I-B^{-T}_m A) \dots (I-B^{-T}_1 A) (I- B^{-1}_0 A) (I-B^{-1}_1 A) \dots (I-B^{-1}_m A),
\end{array}
\end{equation*}
for which $\|E_{composite}\|_A \le \varrho$.

\begin{algorithm}
  \caption{Constructing a Solver with Desired Convergence Factor}
  \label{alg:adaptive-ml}
  \begin{algorithmic}[1]
    \Require $A\in\mathbb{R}^{n\times n}$ s.p.d.\ matrix;
             $B_0$ ($A$–convergent s.p.d.\ pre-solver);
             $m$ (tester iterations);
             $\varrho\in(0,1)$ (target factor)
    \Ensure Solver $B$ with convergence factor $\le \varrho$
    \State $B \gets B_0$
    \State $\varrho_B \gets 1$
    \For{$k \gets 0,1,\dots$ \textbf{until} $\varrho_B \le \varrho$}
      \State $\bx_0 \gets$ random vector
      \For{$s \gets 1,2,\dots,m$}
        \State $\bx_s \gets (I-B^{-1}A)\bx_{s-1}$
      \EndFor
      \State $\varrho_B \gets \lVert\bx_m\rVert_A \big/ \lVert\bx_{m-1}\rVert_A$
      \State $\bw \gets \bx_m / \lVert\bx_m\rVert$
      \State Construct $B_k$ and $B_k^\top$ using $\bw$ (Sec.~\ref{section: modularity based AMG})
      \State $B \gets$ symmetric composition of $B$ and $B_k$ (Eq.~\ref{symmetric composition of two solvers})
    \EndFor
    \State \Return $B$
  \end{algorithmic}
\end{algorithm}

\section{Modularity coarsening based AMG}\label{section: modularity based AMG}
Given a graph $G$ with a vertex set $\N=\{1,2,\;\dots,\;n\}$ and an edge set $\E \subset \N \times \N$ of undirected pairs of vertices $e=(i,j)$. We assume that we have given weights $w_e$ assigned to each edge $e$. In this setting, there is a one-to-one mapping between weighed graphs and symmetric sparse matrices $A = (a_{ij})^n_{i,j=1}$. The edges of such a graph are the pairs $(i,j)$ for which $a_{i,j} \ne 0$. We can also let $w_e = a_{i,j}$ accepting that some ``weights'' can be negative.

In what follows, we assume that the rowsums of $A$ are positive, i.e., 

\begin{equation}\label{weighetd vertex degree}
r_i =\sum\limits_j a_{i,j} > 0.
\end{equation}
Let $\br = (r_i)^n_{i=1} \in {\mathbb R}^n$ and $\bone = (1) \in {\mathbb R}^n$ be the unity constant vector. 
We have, $\br = A \bone$
Finally, let $T = \sum\limits_i r_i = \bone^T A \bone$ be the total rowsum of $A$.
The following matrix is referred to as the {\em modularity matrix} \cite{Newman}
\begin{equation*}
    B = A - \frac{1}{T}\br \br^T.
\end{equation*}
By construction, we have that
\begin{equation}\label{zero row sums of B}
B\bone =0.
\end{equation}
Based on its entries $b_{ij} = a_{ij} - \frac{r_ir_j}{T}$ one defines the so-called modularity functional (cf., \cite{Newman}) associated with a partitioning  of the vertex set into non-overlapping groups $\{\A\}$ of vertices, referred as to clusters, communities, and in the AMG setting as to aggregates.

The modularity functional reads
\begin{equation*}
    Q = \frac{1}{T}\sum\limits_\A \sum\limits_{i,j \in \A} b_{ij}
    = \frac{1}{T}\sum\limits_\A \sum\limits_{i,j \in \A} \left(a_{ij} - \frac{r_ir_j}{T}\right).
\end{equation*}

Different partitioning (sets of aggregates) gives different values for $Q$. The ones that lead to larger values of $Q$ are viewed to be better in terms of detecting {\em  community} structures in graphs (in social networks in particular). The functional $Q$ gives a measure how strongly are the vertices in a community connected versus the connections across the communities. The strength of connectivity is also a concept adopted in the  classical AMG methods (e.g., \cite{amg1, amg2, Ruge_Stuben_87__MR972756}), which motivated us to adopt it in our coarsening algorithm for general matrices. 
More specifically, in our coarsening strategy we employed the entries $b_{ij}$ of the modularity matrix as edge weights to measure if two vertices are strongly connected. This concept we have preliminary studied in \cite{Quiring2019} and showed that it allowed to detect strong anisotropy. Many other options for edge weights are possible and have been utilized in the past with some success to detect strong anisotropy for example (cf., \cite{smoothed_aggr_MR1393006, N-10, DV2013, DFV2018}) and there is not an ultimate answer as to which one is preferred.  

We note that in the actual coarsening only the positive entries $b_{ij}$ matter. Since, by assumption, we consider matrices with positive rowsums $r_i$, the latter means that for each vertex $i$, there is at least one adjacent vertex $j$ for which $a_{ij} > 0$. Since $b_{ij} < a_{ij}$
then $b_{ij}$ can be positive only for vertices $i,j$ that correspond to the sparsity graph of $A$. Finally, since $B \bone =0$ (equation \ref{zero row sums of B}), i.e., $B$ has zero rowsums, this shows that for each $i$ there is at least one  adjacent vertex $j$ for which $b_{ij} \ge 0$, that is at least one edge $e =(i,j)$ associated with $i$ has a positive modularity weight $b_{ij}$. 

In the actual coarsening we implemented, we adopted a version of the parallel Luby's weighted  matching algorithm (\cite{Luby1986, JonesPlassmann1993}) that creates a set $\M$ of pairwise matching (each viewed as a set of two vertices). Namely, given an edge $e$ and its immediate neighbors $e'$ (i.e., ones that share a common vertex with $e$), we compare their edge weights $w_{e'}$ with the weight $w_e$ of $e$. In our case, we have $w_e = b_{ij}, \; e=(i,j)$.
More specifically, if
\begin{equation*}
w_e > \max\limits_{e{'} \ne e,\; e{'} \text{ neighbor of }e} w_{e'},
\end{equation*}
we choose $e$ to be a member of the matching set $\M$. 
This decision is completely parallel (it does not interfere with the decision made for any other edge). 
The aggregates that we construct by the above decision, are pairwise ones, the edges from $\M$ that have locally maximal edge weights, plus the single vertex aggregates, namely the vertices that have not been matched.
This gives a piecewise constant interpolation matrix $P$, which represents the relation {\em vertex }$i$ is contained in aggregate $i_c$, i.e., the 0-1 matrix with rows corresponding to vertices and columns to aggregates with non-zero entries $1$ at position $(i,i_c)$ if vertex $i$ is contained in aggregate $i_c$. 
Then we perform a coarsening step by forming the matrix $B_c = P^T BP$.
This is done by first computing $A_c = P^TAP$ ( a sparse matrix) and 
$b_c= P^T \br$. Then, $B_c= P^TBP = A_c - \frac{1}{T}\br_c \br^T_c$.
If we introduce the coarse unity vector $\bone_c= (1) \in {\mathbb R}^{n_c}$, it is easily seen that $B_c \bone_c = 0$, i.e., it also has the 
zero row-sum property. This is the case since $\bone = P \bone_c$ (i.e., $P$ interpolates the coarse constant vector to the fine constant vector).
This is the case since each vertex $i$ belongs to unique aggregate, that is, there is only one nonzero entry (equal to 1) in each row of $P$.

That is, $B_c$ is actually, the modularity matrix of the coarse graph corresponding to the relation  {\em aggregate}\_{\em aggregate}. The latter means, two aggregates are related if there is an edge $e = (i,j)$ such that $i$ is in one of the aggregates and $j$ is in the other aggregate.
In other words, the coarse graph (with vertices from the aggregates and edges from the pairs of related aggregates) corresponds to the sparsity pattern of the coarse matrix $A_c = P^TAP$. 

One final detail is that one step of parallel matching may not be sufficient to get a reasonable coarsening factor. To achieve a desired coarsening factor, we apply recursion, i.e., Luby's algorithm to the created (intermediate) coarse graphs. 

More details about the modularity matrices and coarsening are summarized in the report \cite{Quiring2019}.

Our main observation in order to utilize the near-null component vector $\bw$ that we compute during the composite adaptive process applied to a general s.p.d. matrix $A$ and using modularity weights is as follows.
Since $A \bw \approx 0$ componentwise by construction, we have for each $i$
\begin{equation*}
    0 \approx w_i  \sum\limits_j a_{ij} w_j,
    \end{equation*}
    or equivalently 
    \begin{equation*}
    0 \le a_{ii} w^2_i \approx \sum\limits_{j\ne i} (-w_i a_{ij} w_j).
    \end{equation*}
That is, the following auxiliary matrix (adjacency matrix of the strength of connection graph) ${\overline A} =({\overline a}_{ij})$ that has non-zero entries ${\overline a}_{ij} = - w_i a_{ij} w_j$ has the positive row-sum  property, for which we use the modularity weights in our coarsening process to create the hierarchy of aggregates and respective $P$ matrices.

\subsection{SA-AMG}
Once the aggregation is obtained in the form of a node-to-aggregate matrix $P$ we can proceed to create a component in the standard Smoothed Aggregation AMG (SA-AMG) fasion \cite{disney_cloth, sa_amg}. Algorithm \ref{alg:adaptive-ml} can be slightly modified to search for a near-null space of desired dimension (instead of a single vector $\bw$) by solving for multiple random starting guess simultaneously and occasionally orthogonalizing them with the modified Gram Schmidt process. The resulting near-null basis acts as the candidate vectors to be used in the SA-AMG. The only other modification needed is in the modularity aggregation algorithm we now start with aggregates corresponding to the finer level's near-null space dimension instead of singletons. We use weighted Jacobi (or block Jacobi in the multiple candidate scenario) with a smoothing weight of $2/3$ to generate the smooth interpolation operators: $\widetilde P := (I - 0.66 D^{-1} A) P$. Then the coarse level is generated by the standard Galerkin assembly $A_c = \widetilde P^\top A \widetilde P$.

\section{Numerical tests}\label{section: experiments}
Various types of numerical experiments are detailed in this section. First, we verify the claim that the adaptive process (Algorithm \ref{alg:adaptive-ml}) can identify useful near-null components. A well known approach for preconditioning linear elasticity discretizations is to use the rigid body modes as SA-AMG candidate vectors. Experiment \ref{exp:elast-rbms} demonstrates the recovery of these modes. Experiments \ref{exp:anisotropy} and \ref{exp:spe10} apply the composite adaptive methods to diffusion problems on 2D and 3D domains. The first is a study with an anisotropic coefficient and the second with a coefficient from the SPE10 \cite{spe10} benchmark permeability data with extreme heterogeneity ($>10^{9}$ permeability ratios on adjacent cells). Finally, we include a test for a non-PDE sparse matrix from the suitespare \cite{suitesparse} collection called \verb|G3_circuit| submitted by Advanced Micro Devices, which comes from a circuit simulation problem.

All convergence plots are paired with at least one table with details about the associated matrices, solver parameters, and resulting solver. Table \ref{tab:results_notation} contains abbreviations and definitions needed to understand these results.

\renewcommand{\arraystretch}{1.2}
\begin{table}[htbp]
  \centering
  \caption{Notation used in convergence plots and associated tables.}
  \begin{tabular}{@{}>{\centering\arraybackslash}p{2cm}p{10cm}@{}}
    Abbreviation & Definition / description \\ \hline
    $k$                  & Number of components in the composite solver. \\[2pt]

    nnz                  & Number of nonzeros. \\[2pt]

    DoFs                 & Degrees of Freedom (length of solution vector including fixed boundary values). \\[2pt]

    $h$                  & Number of geometric mesh refinements. \\[2pt]

    $\mu$                & Cycle parameter: \(\mu=1\) (V-cycle), \(\mu=2\) (W-cycle),
                           \(\mu>2\) (generalised W-cycles). \\[6pt]

    $\nu$                & \textbf{Smoothing steps} (total pre- + post-relaxations per level). \\[2pt]

    $\gamma$             & \textbf{Coarsening factor}: ratio of unknowns on adjacent grids. \\[4pt]

    $N_{\mathrm{SA}}$    & \textbf{SA candidates}: number of vectors used to build the
                           smoothed-aggregation prolongation operator. \\[2pt]

    $C_k$                & Average \emph{operator complexity} per component:
                           \(\displaystyle
                             C_k = \tfrac{1}{k} \cdot \tfrac{\text{nonzeros in all hierarchies}}{\text{nonzeros in fine matrix}}
                           \). \\[6pt]

    $\br^{(i)}$            & Residual after the \(i^{\text{th}}\) composite iteration:
                           \(r^{(i)} = \mathbf{w} - A\,\mathbf{x}^{(i)}\). \\[4pt]

    $\dfrac{\|\br^{(i)}\|}{\|\br^{(0)}\|}$  
                         & Relative residual at iteration \(i\)
                           (or after \((2k-1)i\) cycles). \\[6pt]

    $\rho$               & Asymptotic convergence factor per cycle,
                           \(\displaystyle
                             \rho = \left(
                               \frac{\|\br^{(i)}\|}{\|\br^{(i-1)}\|}
                             \right)^{\!\!\frac{1}{2k-1}}
                           \).  \\[4pt]

    Stat.                & $\mu$-cycles required for a composite operator to reach $10^{-12}$ relative residual as a stationary iteration \\[2pt]

    PCG                  & $\mu$-cycles required for a composite operator to reach $10^{-12}$ relative residual as a preconditioner for conjugate gradient\\[2pt]

    \multicolumn{2}{@{}p{\dimexpr2cm+11cm}@{}}{\small
      \emph{Notes about convergence plot tables:
      Triangles mark true composite iterations (\(2k-1\) cycles each). Left panels show stationary and right panels show PCG convergence. Each row down is a new refinement level.}} \\
  \end{tabular}
  \label{tab:results_notation}
\end{table}

\subsection{Recovering RBM for Linear Elasticity}
\label{exp:elast-rbms}
A common usage of SA-AMG is to utilize the rigid body modes (RBM) of the linear elasticity PDE as the candidate vectors. A reasonable goal for our adaptive algorithm would be to either recover these modes as candidates or construct a multigrid operator with better convergence than the SA-AMG with RBM candidates. 

Let $\{m_i\}_{i=1}^6$ be the $6$ normalized RBMs (XY, YZ, ZX rotations and X, Y, Z translations) associated with a 3d linear elasticity PDE and let $Q$ be an orthonormal basis for these modes. Also let $W$ be an orthonormal basis for a set of SA-AMG candidate vectors computed through Algorithm \ref{alg:adaptive-ml}. We score how much of the RBM space is covered by the candidate space with a value from $0$ to $1$ with 
$$\frac{\|Q^\top W\|_{\star}}{6}$$
where $\|\cdot\|_\star$ is the nuclear norm. We also score each individual RBM by projecting into the candidate space with $\|W^\top m_i\|_2$.

\begin{figure}[htbp]
  \centering 
  \textbf{Recovery of the Rigid Body Modes}
    \begin{tabular}{cc|ccc|ccc}
        $N_{\mathrm{SA}}$ & \textbf{Score} & \multicolumn{3}{c|}{\textbf{Rotations}} & \multicolumn{3}{c}{\textbf{Translations}} \\
        & & \textbf{XY} & \textbf{YZ} & \textbf{ZX} & \textbf{X} & \textbf{Y} & \textbf{Z} \\
        \hline
         3 & 0.46 & 0.85 & 0.63 & 0.86 & 0.85 & 0.85 & 0.79 \\
         6 & 0.80 & 0.99 & 0.91 & 0.99 & 0.80 & 0.91 & 0.90 \\
        10 & 0.89 & 1.00 & 0.95 & 1.00 & 0.91 & 0.94 & 0.94 \\
        20 & 0.96 & 1.00 & 0.97 & 1.00 & 0.98 & 0.98 & 0.98
    \end{tabular}
\label{tab:rbm-recovery}
\end{figure}

By applying the adaptivity algorithm, we find that we do `discover' the space spanned by the RBMs in the finest level adaptive SA-AMG candidates. Table \ref{tab:rbm-recovery} provides how much overlap there is between the RBM space and the span of the adaptive candidate when testing on \href{https://mfem.org/examples/}{example 2 code from MFEM, \cite{mfem}} using the `beam\_tet' mesh and 2 refinements ($2475$ DoFs with $72933$ nnz).

\subsection{Anisotropic Diffusion}
\label{exp:anisotropy}
An important class of numerical PDE problems which are challenging to solve with standard multigrid is anisotropic diffusion. A typical approach is to use line smoothing and semi-coarsening only in the direction of the anisotropy; but this generally only works when the direction is known and the domain can be discretized in a structured way to align with it. In this section, we demonstrate that the composite adaptive process is robust to discretized anisotropic diffusion problems posed on 2d and  3d domains $\Omega$.

We use a simple but non-trivial 2d problem to demonstrate how the adaptive procedure mimics some of the techniques such as line smoothing and semi-coarsening on unstructured problems in a purely algebraic way. For both the 2d and 3d problems, the PDE we solve is the anisotropic diffusion problem with homogeneous Dirichlet conditions, stated as follows:
\begin{equation} \label{eq:anis_dif}
    \begin{cases} 
	    -\nabla \cdot (K \nabla u) = 1, & \text{in}\; \Omega\\ 
	    u = 0, & \text{on}\; \partial \Omega
    \end{cases}
\end{equation}
where $K = \epsilon I + \bbeta \bbeta^T$ with 
$$\begin{cases} 
	\bbeta=\begin{bmatrix} \cos \theta \\ \sin \theta \end{bmatrix} & \text{when}\; \Omega \subset \mathbb{R}^2\\
		\bbeta=\begin{bmatrix} \cos \theta \cos \phi \\ \sin \theta \cos \phi \\ \sin \phi \end{bmatrix} & \text{when}\; \Omega \subset \mathbb{R}^3
\end{cases}$$ 
is the dominant direction of the diffusion tensor. We use $\epsilon=10^{-6}$ in the following tests.

These examples were assembled using MFEM using the lowest order $H^1$-conforming finite element space. For the first example, we use the 2d `star.mesh' (MFEM sample mesh, \cite{mfem}), geometrically refined until the matrix has $1,313,281$ DoFs.

First, we verify the results from \cite{Quiring2019}, demonstrating the modularity matching algorithm we use for partitioning effectively identifies the direction of anisotropy and creates aggregation based hierarchies which align with it. When using a linear continuous Galerkin finite element discretization ($p=1$), each DoF in the assembled matrix is associated with a vertex in the mesh of the geometry.

\begin{figure}[htbp]
  \centering
  \textbf{Algebraically Smooth Error and Resulting Aggregation}
  \setlength{\tabcolsep}{0.5pt}
  \begin{tabular}{cc}
    \includegraphics[trim={14cm 0 12cm 0},clip,width=.49\textwidth]{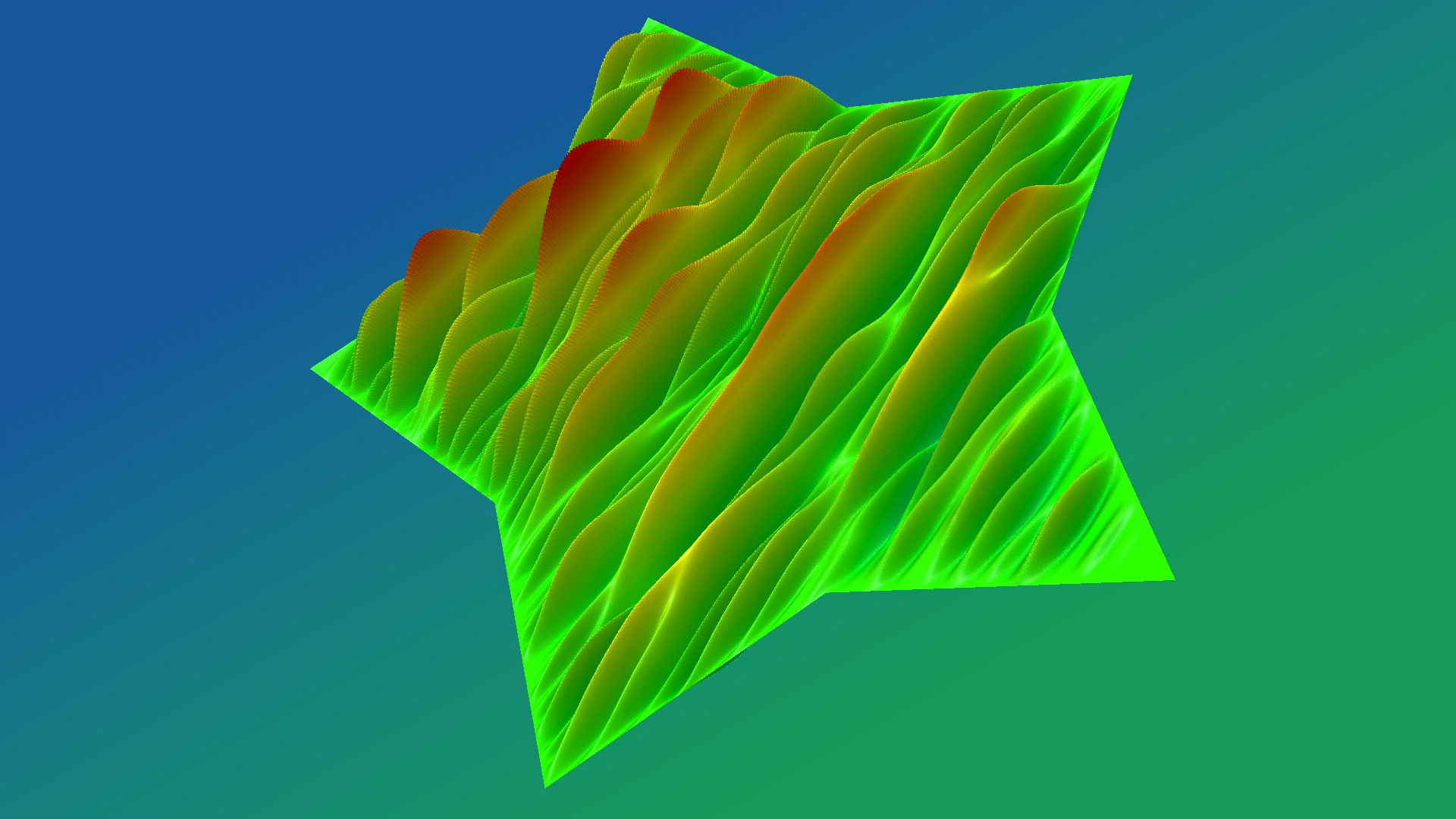} &
    \includegraphics[width=.49\textwidth]{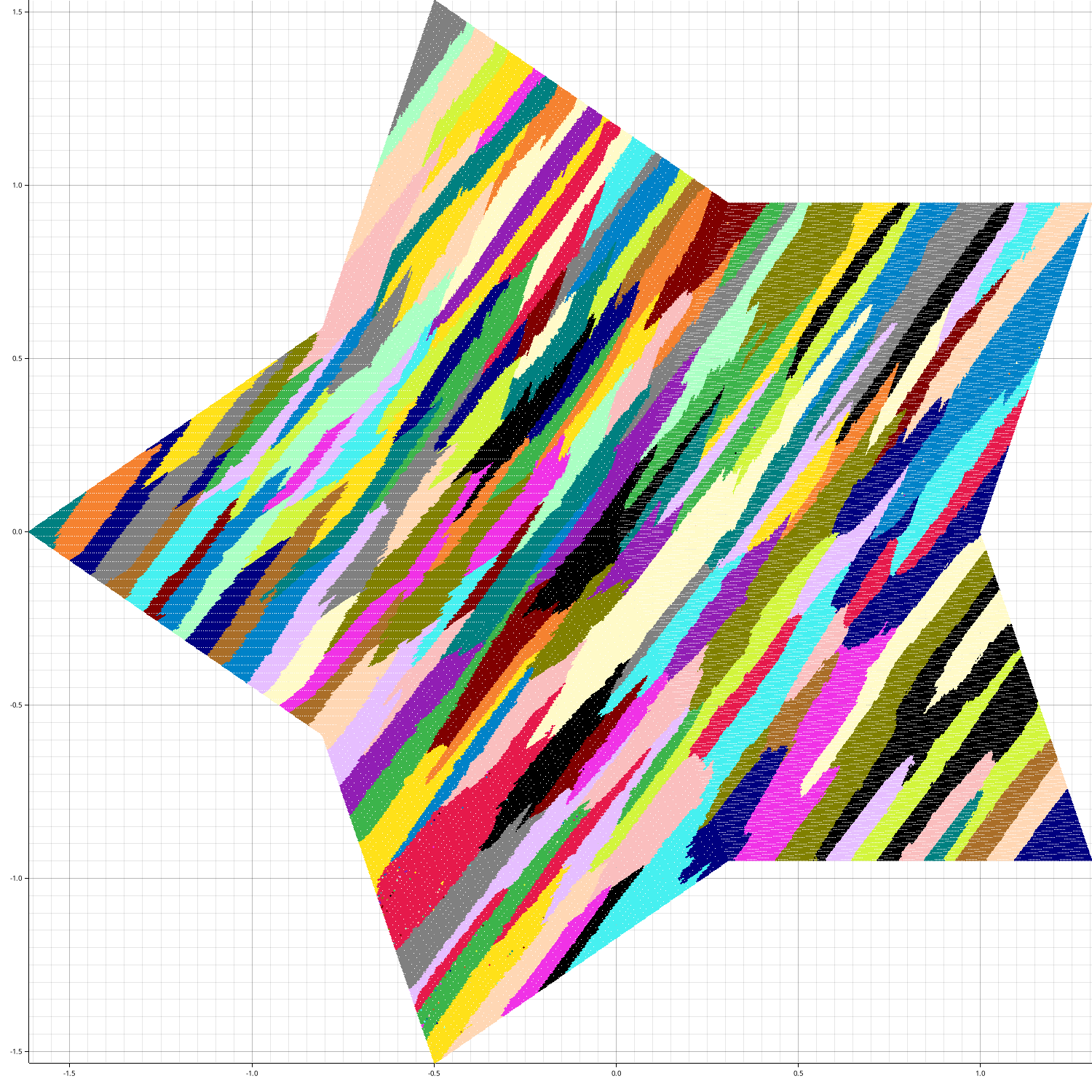}
  \end{tabular}
	\caption{Left: near-null grid function; generated by testing a solver on the homogeneous system as described in Section \ref{section: the adaptive process}. Right: `coarse grid' aggregation generated from the modularity matching algorithm applied to the `strength of connection' graph associated with this smooth error (described in Section \ref{section: modularity based AMG}).}
  \label{fig:2d_viz}
\end{figure}

 In Figure \ref{fig:2d_viz}, we color each vertex in the mesh according to which aggregate they belong after the modularity matching algorithm is applied using the associated algebraically smooth error vector generated by the `tester' phase of the adaptivity algorithm. As expected, algebraically smooth error varies slowly in the direction of anisotropy but quickly orthogonal to it. Additionally, notice that the coarse grid aggregates align nicely with the anisotropy. This demonstrates that our algorithm and implementation provide an algebraic approach akin to structured semi-coarsening techniques for this problem. 

\begin{figure}[htbp]
  \centering 
  \textbf{2d Adaptive Process Convergence}

  \vspace{0.5em}
  \begin{tabular}{cccccccc}
	  $h$ & DoFs & nnz & $N_{\mathrm{SA}}$ & $\gamma$ & $\mu$ & $\nu$ & $C_k$\\ 
	  \hline
	  $7$ & $1,313,281$ & $11,804,161$ & $3$ & $8$ & $1$ & $1$ & $1.8$
  \end{tabular}
	\vspace{0.5em}

  \setlength{\tabcolsep}{0.5pt}
  \begin{tabular}{cc}
    \includegraphics[width=.49\textwidth]{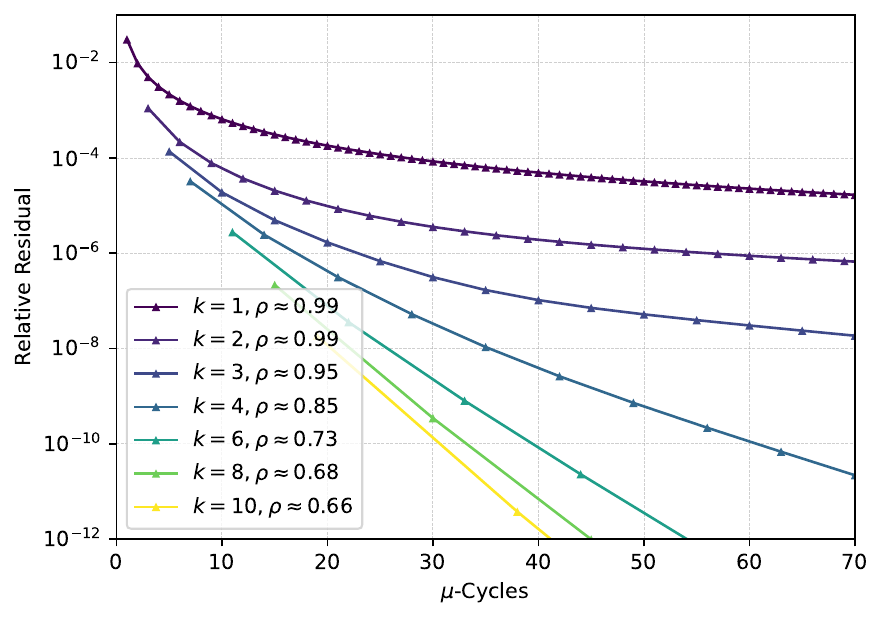} &
    \includegraphics[width=.49\textwidth]{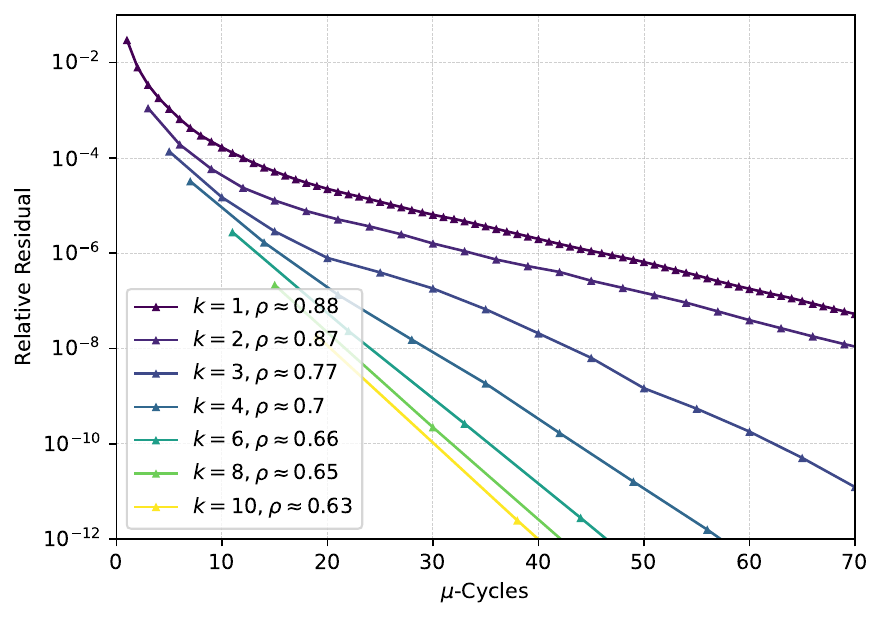}
  \end{tabular}
  \caption{Convergence rates of adaptive solvers with varying components. Left: stationary solver. Right: preconditioned conjugate gradient.}
  \label{fig:2d_convergence}
\end{figure}

Next, in Figure \ref{fig:2d_convergence}, we present the performance of the adaptive solver as more components are added. After a few components, the solver improves quickly and around 6 components the improvements diminish. Notice that at 10 components the convergence rates for the stationary method is as fast as when used as a preconditioner for conjugate gradient; reinforcing our claim in Section \ref{section: the adaptive process} that the adaptive algorithm resembles variational iterative methods.

Next, we investigate the scalability on a simple 3d refinement study. For this study we generate the standard unit cube mesh (again with MFEM API, \cite{mfem}) and solve PDE \ref{eq:anis_dif} using the same discretization process as the 2d experiment. The results of these tests are summarized in Table \ref{tab:3d_scaling} but the full convergence plots akin to Figure \ref{fig:2d_convergence} for each refinement can be found in Figure \ref{fig:3d_full_results} in the Appendix.

\begin{figure}[htbp]
  \centering 
  \textbf{3d Refinement Scaling Study Summary}

  \vspace{0.5em}
  \begin{tabular}{ccccc}
	  $k$ & $N_{\mathrm{SA}}$ & $\gamma$ & $\mu$ & $\nu$ \\ 
	  \hline 
	  $6$ & $6$ & $16$ & $1$ & $2$
  \end{tabular}
  \begin{tabular}{cccccc}
	  $h$ & DoFs & nnz & $C_k$ & Stat. & PCG\\ 
	  \hline
	  $1$ & $9,261$     & $128,581$    & $1.6$ & $16$ & $16$\\
	  $2$ & $68,921$    & $993,961$    & $1.9$ & $31$ & $28$\\
	  $3$ & $531,441$   & $7,815,121$  & $2.2$ & $46$ & $41$\\
	  $4$ & $4,173,281$ & $61,979,041$ & $2.5$ & $75$ & $60$
  \end{tabular}

  \caption{Left: solver parameters. Right: scaling results.}
  \label{tab:3d_scaling}
\end{figure}

We see from the summary in Table \ref{tab:3d_scaling} that there is a linear relationship between refinement level and operator complexity.
This can be mitigated by skipping the smoothing step the finest level interpolation matrix; but this comes at the cost of worsened convergence factor. More sophisticated approaches (not studied here) such as weight pruning may be more effective at reducing complexity without slowing convergence.
Last, the scaling of cycle count looks suboptimal, but the anisotropy ratio of $10^{-6}$ isn't captured on the low refinement discretizations making the convergence on these matrices faster than they should be.

\subsection{SPE10 Benchmark}
\label{exp:spe10}
The SPE10 \cite{spe10} is a $1,122,000$ cell, 3 dimensional, rectangular Cartesian grid which porosity data for each cell. Again, we solve the Diffusion PDE \ref{eq:anis_dif} on this domain but now the coefficient $K$ is a piecewise constant diagonal matrix defined by the provided permeability values shown in Figure \ref{fig:spe10-coef}. The extreme heterogeneity of the coefficient and complex sub-surface structures present challenges for standard multiscale methods. An effective multigrid method for this problem has to avoid destroying this structure in the coarse representations.

\begin{figure}
  \centering 
  \textbf{SPE10 Coefficient Structure}

  \includegraphics[width=0.49\textwidth]{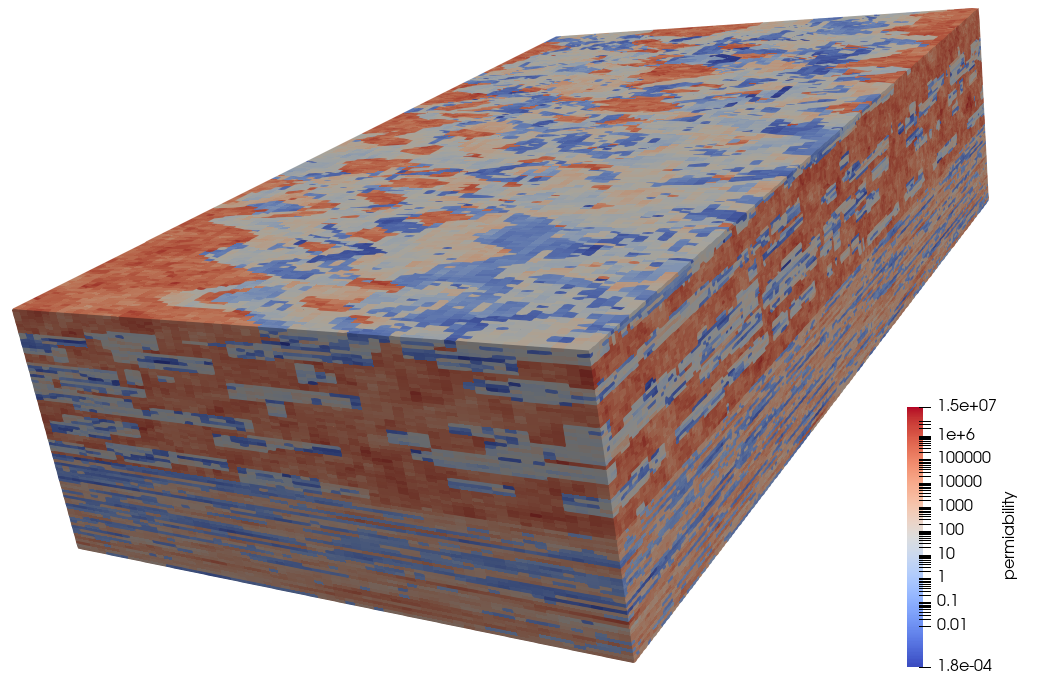}
  \includegraphics[width=0.49\textwidth]{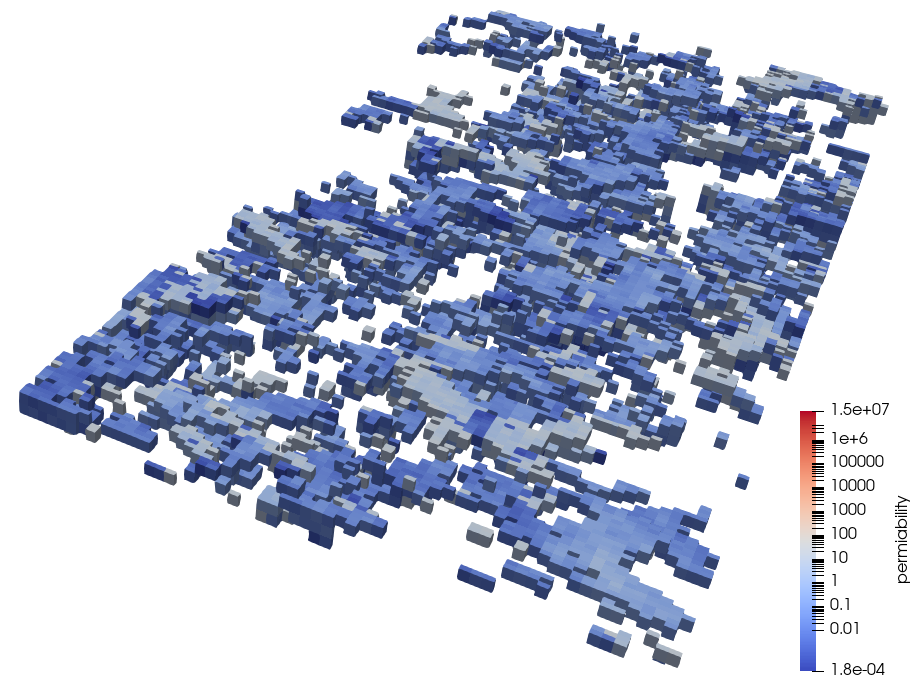}
  \caption{Left: full mesh with permiability coefficient magnitude visualized. Right: thin slice from the upper region of the geometry with highly permiable areas clipped to show complex structure.}
  \label{fig:spe10-coef}
\end{figure}

\begin{figure}[htbp]
  \centering 
  \textbf{SPE10 Adaptive Process Convergence}

  \vspace{0.5em}
  \begin{tabular}{cccccccc}
    $h$ & DoFs & nnz & $N_{\mathrm{SA}}$ & $\gamma$ & $\mu$ & $\nu$ & $C_k$\\ 
	  \hline
	  $0$ & $1,159,366$ & $30,628,096$ & $3$ & $16$ & $1$ & $1$ & $1.5$
  \end{tabular}
	\vspace{0.5em}

  \setlength{\tabcolsep}{0.5pt}
  \begin{tabular}{cc}
    \includegraphics[width=.49\textwidth]{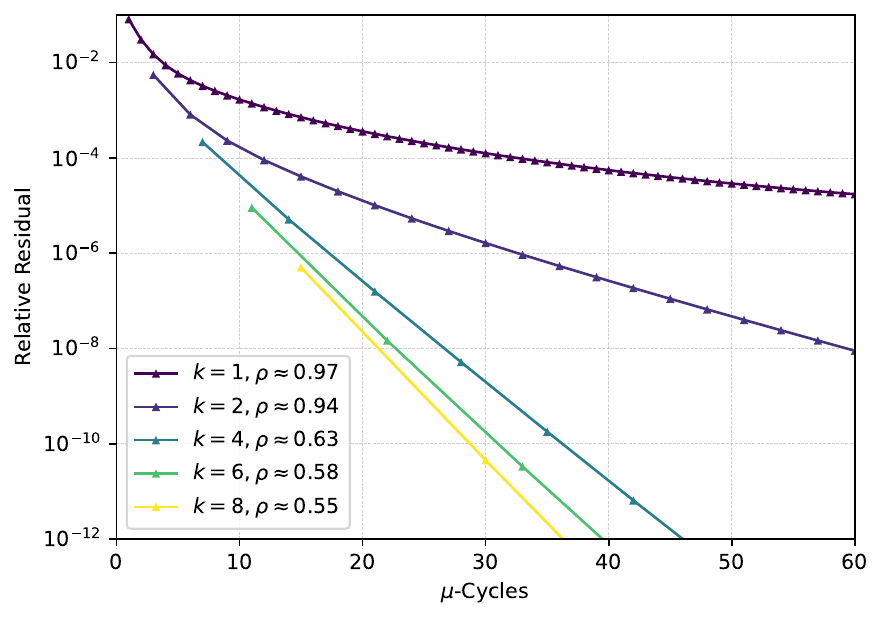} &
    \includegraphics[width=.49\textwidth]{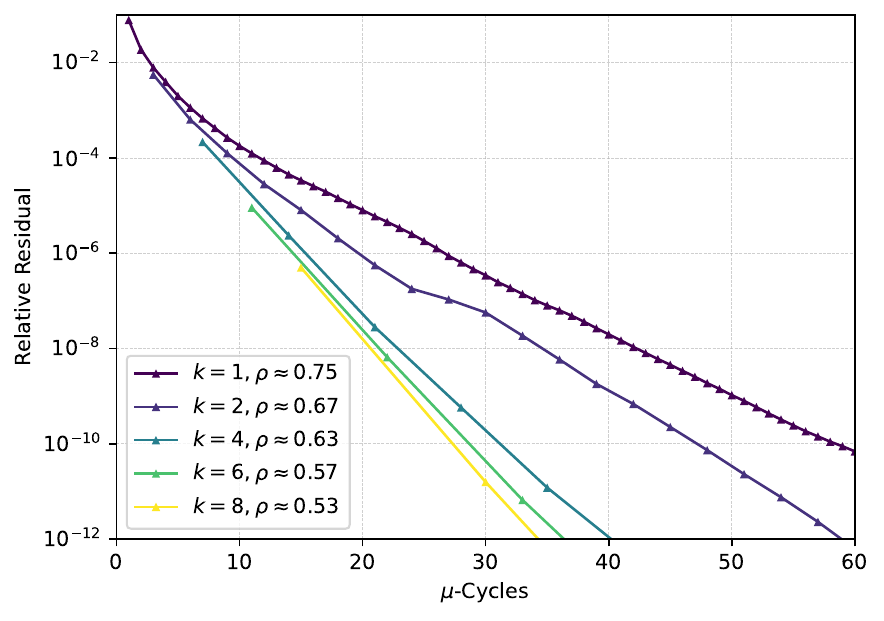}
  \end{tabular}
  \caption{Convergence rates of adaptive solvers with varying components. Left: stationary solver. Right: preconditioned conjugate gradient.}
  \label{fig:spe10_convergence}
\end{figure}

This benchmark isn't nearly as challenging to solve as the anisotropy problems but the convergences shown in Figure \ref{fig:spe10_convergence} demonstrate that this problem can still benefit from adaptivity. With only $3$ SA-AMG candidates each component the bulk of the adaptive improvements are realized with just $4$ components and after $8$ components the convergence rate of the stationary method is the same as preconditioned conjugate gradient.

\subsection{Suitsparse (AMD G3 Circuit)}
\label{exp:g3}
The final experiment we present demonstrates that the adaptivity process works on SPD problems which don't come from PDE discretizations. The matrix studied here comes from the suitesparse \cite{suitesparse} collection called \verb|G3_Circuit| submitted by AMD coming from a circuit simulation problem.

\begin{figure}[htbp]
  \centering 
  \textbf{G3 Circuit Adaptive Process Convergence}

  \vspace{0.5em}
  \begin{tabular}{ccccccc}
    DoFs & nnz & $N_{\mathrm{SA}}$ & $\gamma$ & $\mu$ & $\nu$ & $C_k$\\ 
	  \hline
	  $1,585,478$ & $7,660,826$ & $3$ & $16$ & $1$ & $1$ & $1.7$
  \end{tabular}
	\vspace{0.5em}

  \setlength{\tabcolsep}{0.5pt}
  \begin{tabular}{cc}
    \includegraphics[width=.49\textwidth]{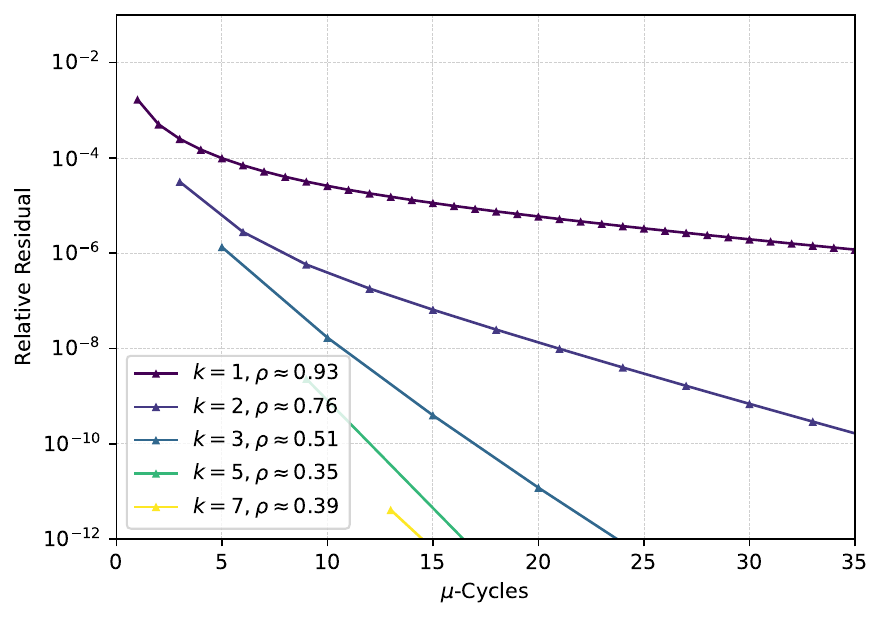} &
    \includegraphics[width=.49\textwidth]{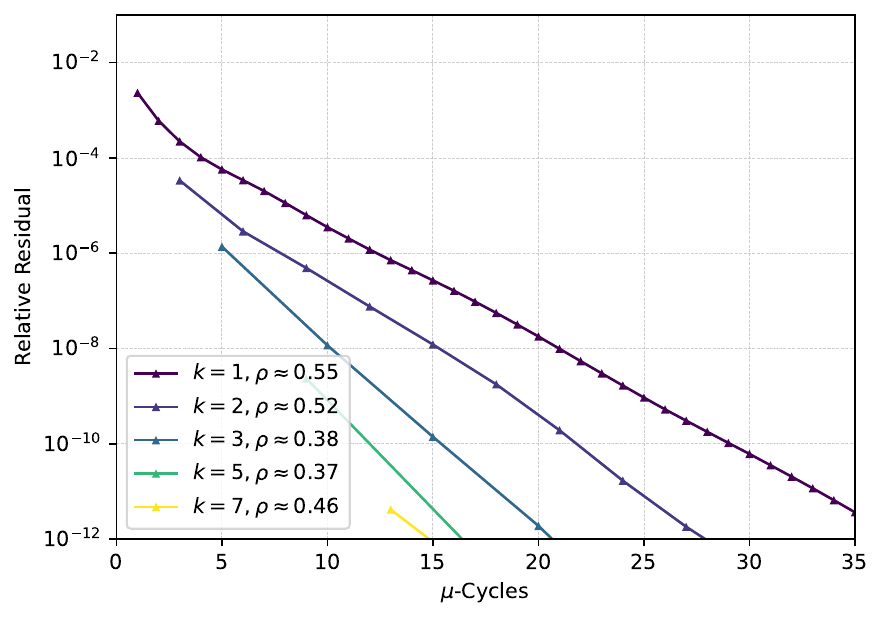}
  \end{tabular}
  \caption{Convergence rates of adaptive solvers with varying components. Left: stationary solver. Right: preconditioned conjugate gradient.}
  \label{fig:g3_convergence}
\end{figure}

Similar to the SPE10, this is a less challenging test matrix. Figure \ref{fig:g3_convergence} shows that with only $3$ SA-AMG candidates the adaptive improvements are realized with just $3$ components and after $5$ components the convergence rate of the stationary method is the same as preconditioned conjugate gradient. At $7$ components, there are machine precision related numerical issues after a single application as a stationary method indicating the $7$ component operator is nearly equivalent to $A^{-1}$.

\section{Conclusions}\label{section: conclusions}
In this paper, we provided a theoretical justification of utilizing near-null components of the preconditioned matrix $B^{-1}A$ for a number of solvers $B$ in the construction of new AMG-based hierarchies. The result is that these vectors are actually near-null components of the original matrix $A$ (and not only of the preconditioned one $B^{-1}A$). The theoretical findings are illustrated with numerical tests. 
We also tested a graph-modularity based approach for coarsening utilizing a process to create strength of connection graphs parameterized by the identified near-null components. The result is the creation of a sequence of different hierarchies created to target specific slow-to-converge error modes which are combined in a symmetric composite manner to create an adaptive preconditioner. 

The numerical experiments demonstrate that the developments to adaptive SA-AMG provide a promising solution to problems for which no scalable solvers exist, especially in the case of extremely anisotropic elliptic PDE discretizations. The second notable takeaway from the experiments is that as many components are added to the composite adaptive operator the convergence rate as a stationary iteration approaches the convergence rates obtained using the same operator as a preconditioner for variational methods such as Krylov acceleration.

Two main challenges are also apparent from the results of the experiments. First, the time to set up the solver is quite extreme and for especially degenerate systems many near-null candidate vectors may be required to get good convergence. Techniques to speed up this process may be available but these composite methods will only be practical in scenarios where the setup cost can be amortized over many linear solves or if the system is changing slowly between solves (potentially in a time integration scheme) then much of the setup can likely be reused. Second, there is a large memory requirement when multiple hierarchies with large operator complexities are required. This may be quite prohibitive in a large scale application and other techniques such as pruning interpolations may be required to control the complexities and/or once the near-null components are identified one may recreate a new SA-AMG hierarchy using the multiple components simultaneously (as in the original SA-AMG methods) similarly to the approach taken in \cite{DV2019}.

Finally, we should mention that the algorithms developed in this study are implemented in Rust, a modern systems programming language renowned for its efficacy in low-level implementations of high-performance libraries. Rust's features align well with the demands of complex computational tasks, offering both safety and efficiency. The source code for these implementations is available on GitHub, and can be accessed at \href{https://github.com/aujxn/amg_match}{github.com/aujxn/amg\textunderscore match}.

\bibliographystyle{my_etna_preprint}
\bibliography{references.bib}

\begin{appendix} \label{sec:full_plots}
\begin{figure}[htbp]
  \centering
  \textbf{3d Refinement Scaling Study: Full Convergence Plots}
  
  \vspace{0.3em}
  \begin{tabular}{cccc}
	  $N_{\mathrm{SA}}$ & $\gamma$ & $\mu$ & $\nu$ \\ 
	  \hline 
	  $6$ & $16$ & $1$ & $2$
  \end{tabular}
  \begin{tabular}{cccc}
    $h$ & DoFs & NNZ & $C_k$\\ 
    \hline
    $1$ & $9,261$     & $128,581$    & $1.6$\\
    $2$ & $68,921$    & $993,961$    & $1.9$\\
    $3$ & $531,441$   & $7,815,121$  & $2.2$\\
    $4$ & $4,173,281$ & $61,979,041$ & $2.5$
  \end{tabular}
  \vspace{0.3em}

  \setlength{\tabcolsep}{0.5pt}
  \begin{tabular}{cc}
    \includegraphics[width=.40\textwidth]{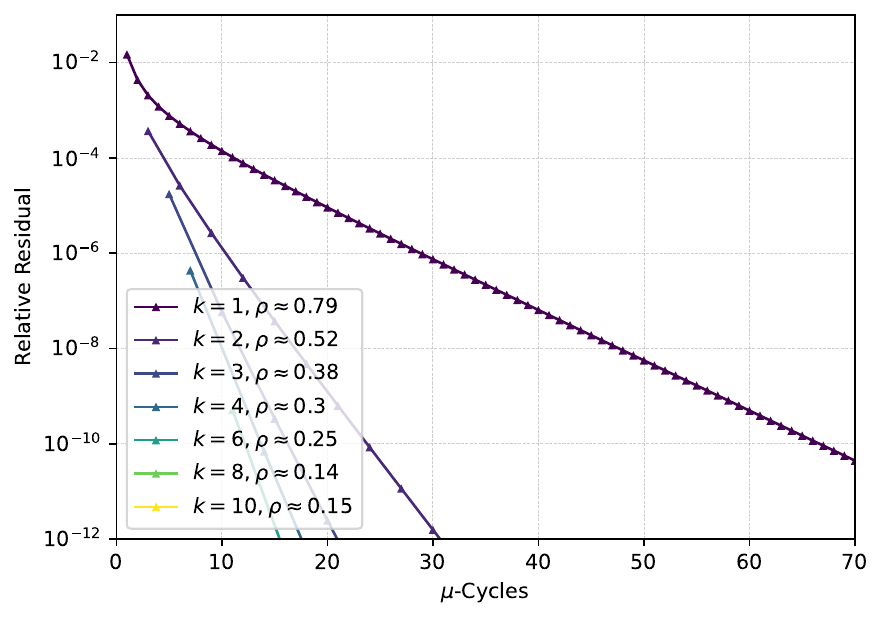} &
    \includegraphics[width=.40\textwidth]{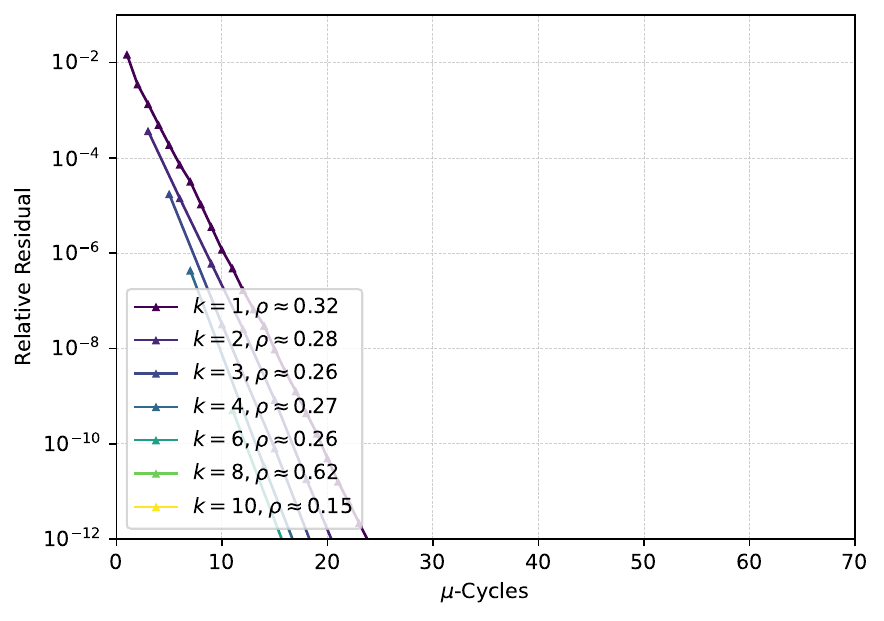}\\
    \includegraphics[width=.40\textwidth]{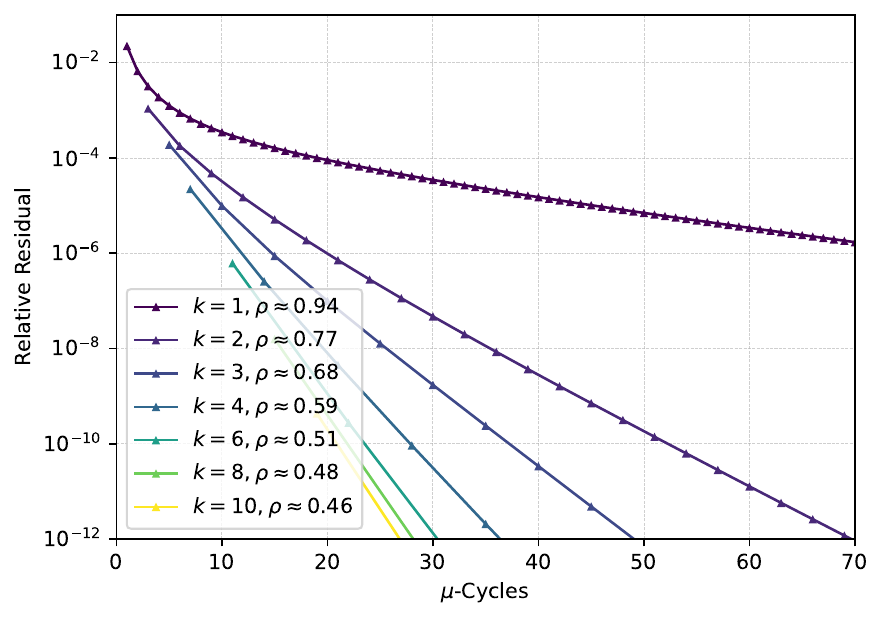} &
    \includegraphics[width=.40\textwidth]{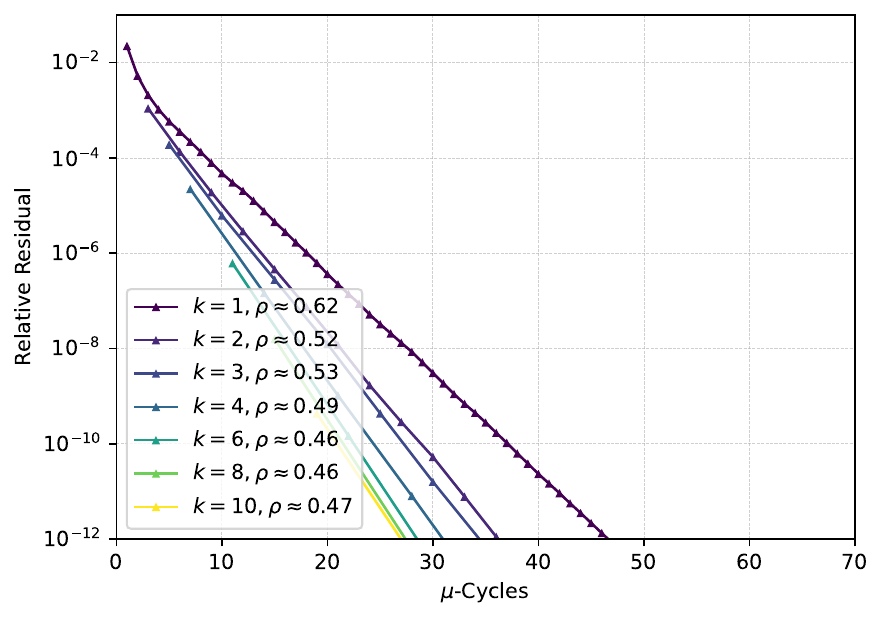}\\
    \includegraphics[width=.40\textwidth]{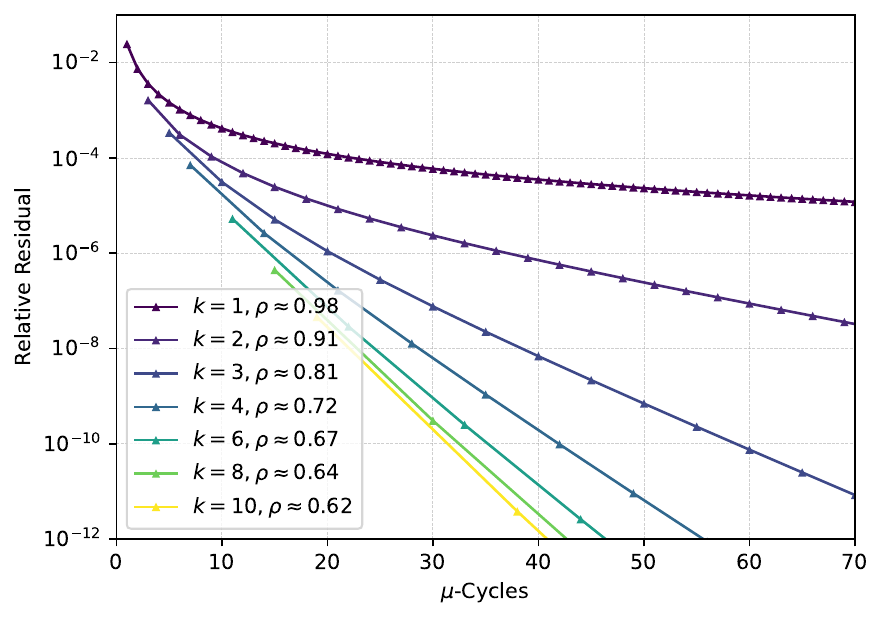} &
    \includegraphics[width=.40\textwidth]{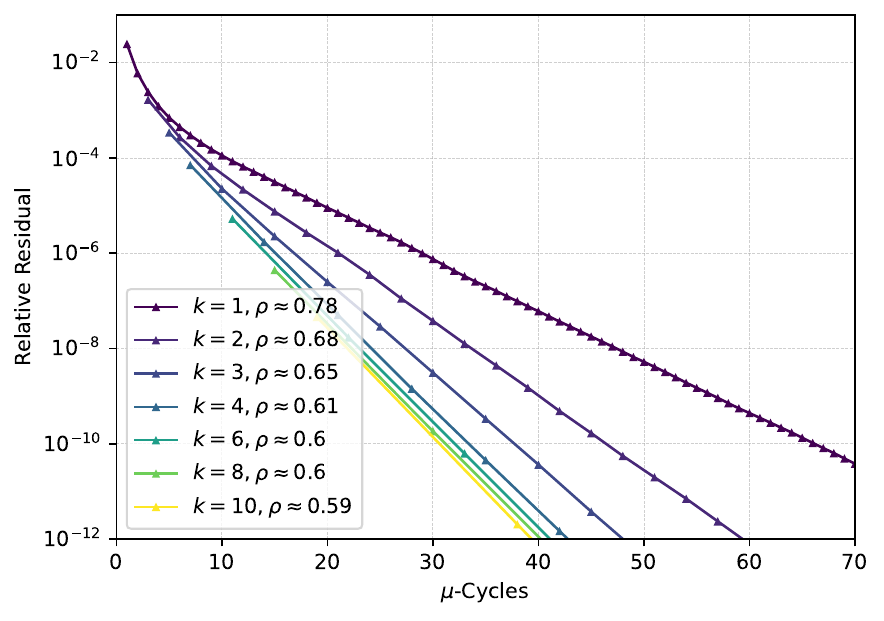}\\
    \includegraphics[width=.40\textwidth]{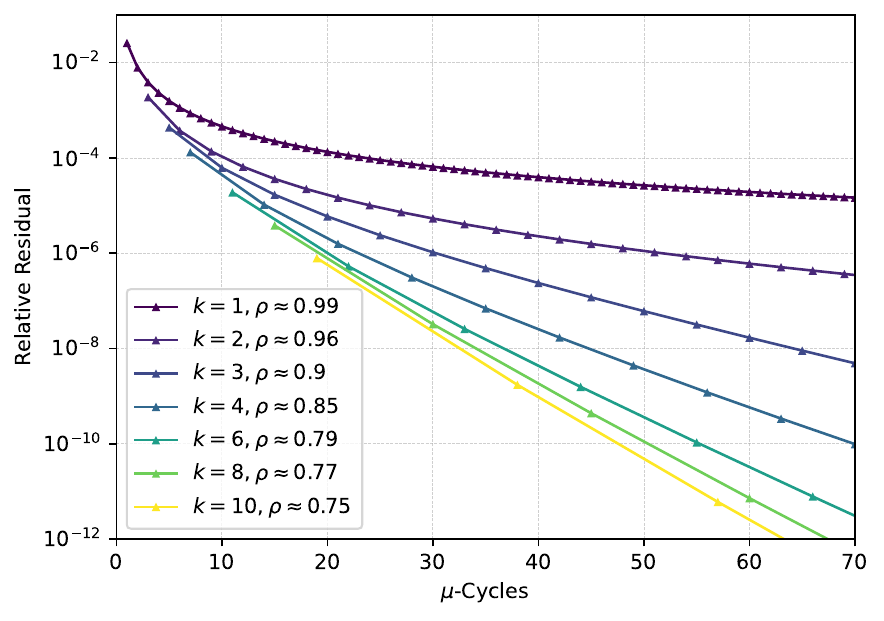} &
    \includegraphics[width=.40\textwidth]{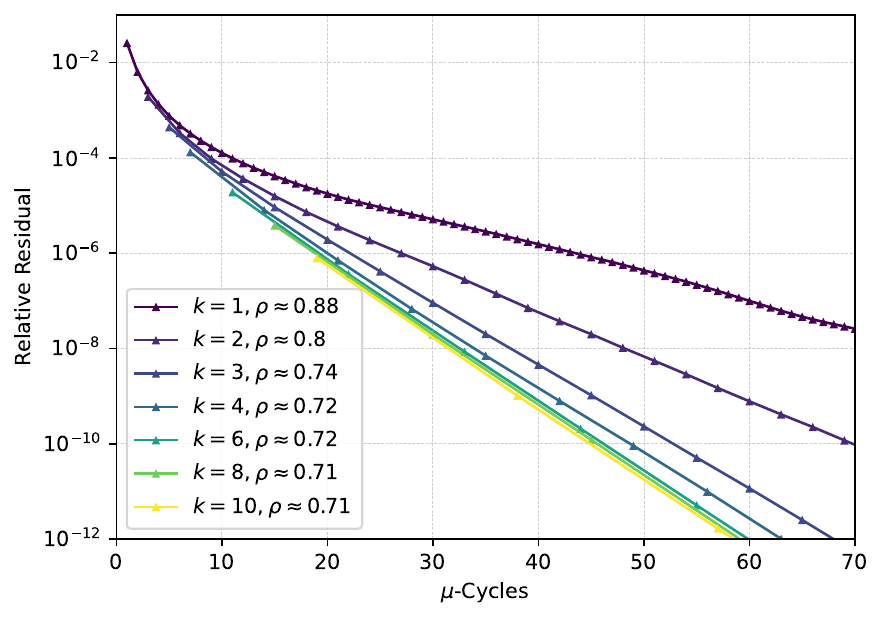}
  \end{tabular}
  \caption{Convergence rates of adaptive solvers with varying components. Left: stationary solver. Right: preconditioned conjugate gradient. Top to bottom: increasing refinement level.}
  \label{fig:3d_full_results}
\end{figure}
\end{appendix}

\end{document}